\documentclass[11pt]{amsart}

\usepackage[T1]{fontenc}
\usepackage[utf8]{inputenc}
\usepackage{lmodern}
\usepackage{amsmath,amssymb,amsthm,mathtools}
\usepackage{enumitem}
\usepackage[margin=1.15in]{geometry}
\usepackage[colorlinks=true,linkcolor=blue,citecolor=blue,urlcolor=blue]{hyperref}

\newtheorem{theorem}{Theorem}[section]
\newtheorem{proposition}[theorem]{Proposition}
\newtheorem{corollary}[theorem]{Corollary}
\newtheorem{lemma}[theorem]{Lemma}
\newtheorem{question}[theorem]{Question}
\theoremstyle{definition}
\newtheorem{definition}[theorem]{Definition}
\newtheorem{remark}[theorem]{Remark}

\newcommand{\codim}{\operatorname{codim}}
\newcommand{\normal}{\mathrel{\trianglelefteq}}

\title{Just-Infinite Loops and Loop Algebras}
\author[T. F. V. Paiva]{Thales Fernando Vilamaior Paiva}
\address{Universidade Federal de Mato Grosso do Sul, C\^ampus de Aquidauana \\ CEP 79200-000, Aquidauana - MS, Brazil }
\email{thales.paiva@ufms.br}

\date{}

\begin{document}

\begin{abstract}
Let $F$ be a field and let $L$ be a loop.  We call $L$ just-infinite if it is infinite and every nontrivial normal subloop has finite index, and we call the possibly nonassociative loop algebra $F[L]$ just-infinite if it is infinite-dimensional and every nonzero two-sided ideal has finite codimension. We first prove that just-infiniteness of $F[L]$ always implies just-infiniteness of $L$.  Next, using the Chein construction, we show for every infinite group $G$ that $M(G,2)$ is just-infinite if and only if $G$ is just-infinite, and that $F[M(G,2)]$ is just-infinite if and only if $F[G]$ is just-infinite.  We extend the algebraic equivalence to the generalized Moufang doubles $M(G,*,g_0)$ whenever $G$ is infinite and nonabelian. Finally, we construct a single locally finite, residually finite, nonassociative Moufang loop $L$ for which $F[L]$ is residually finite-dimensional, locally finite-dimensional, and just-infinite over every field, and we explain why infinite nonassociative RA loops cannot be just-infinite.
\end{abstract}



\maketitle

\section{Introduction}

An infinite group is called \emph{just-infinite} if each of its nontrivial normal subgroups has finite index.  The analogous condition for algebras asks that each nonzero ideal have finite codimension.  This algebraic notion, also called nearly finite-dimensional, has been studied in the associative case by Farkas and Small \cite{FarkasSmall}, Farina and Pendergrass--Rice
\cite{FarinaRice}, and Bell, Farina, and Pendergrass--Rice \cite{BellFarinaRice}, and in the alternative case by Panasenko
\cite{Panasenko}.  The purpose of this paper is to compare the just-infinite property of a loop $L$ with that of its possibly nonassociative loop algebra $F[L]$.

The first implication is simple but useful.  If $N$ is a normal subloop of $L$, the linear extension of the quotient map $q_N:L\to L/N$ is a surjective algebra homomorphism $F[q_N]:F[L]\to F[L/N]$ and therefore induces an isomorphism
\[
  F[L]/\ker F[q_N]\cong F[L/N].
\]
Consequently, just-infiniteness of $F[L]$ implies just-infiniteness of $L$, as stated in Theorem~\ref{teoalgebra-implies-loop}.  The converse fails already for groups, as explained in Remark~\ref{remconverse}.

For inverse-property loops, we make the obstruction to this converse precise in Theorem~\ref{teofaithful-criterion}: the algebra $F[L]$ is just-infinite if and only if $L$ is just-infinite and every nonzero ideal that is faithful on
the basis loop $L$ has finite codimension.  Thus the criterion isolates exactly the ideals that are not detected by the normal-subloop lattice.

The main part of the paper concerns Chein doubles.  We first prove in Theorem~\ref{teoloop-chein} that, for every infinite group $G$, the Moufang loop $M(G,2)$ is just-infinite if and only if $G$ is just-infinite.  At the level of loop algebras, writing $A=F[G]$ and $B=F[M(G,2)]=A\oplus Au$, we describe the doubled ideals $K\oplus Ku$ and prove that every nonzero ideal of $B$ intersects $A$ nontrivially whenever $A$ is prime.  Together with Connell's primeness criterion, this gives the
central equivalence in Theorem~\ref{teochein-main}, $F[M(G,2)]$ is just-infinite if, and only if, $F[G]$ is just-infinite, as well as its stable version under arbitrary field extensions in Corollary~\ref{corstable-standard}.  To the best of our knowledge, the faithful-ideal criterion and the equivalences for Chein doubles established
here have not previously appeared in the literature. 

We then establish the analogous algebraic results for the generalized Moufang doubles $M(G,*,g_0)$ associated with admissible triples, assuming that $G$ is infinite and nonabelian.  Finally, using the groups constructed
by Belyaev, Grigorchuk, and Shumyatsky~\cite{BGS}, we obtain in Theorem~\ref{teouniform-example} a single locally finite, residually finite, nonassociative just-infinite Moufang loop whose loop algebra is locally finite-dimensional, residually finite-dimensional, and just-infinite over every field.  In contrast, Corollary~\ref{corra-obstruction} shows that the
structure of nonassociative RA loops provides an obstruction: no infinite nonassociative RA loop, and hence none of its loop algebras, can be just-infinite.

We conclude by formulating several questions concerning faithful ideals, generalized doubles at the loop level, and finitely generated examples.  These questions arise naturally from the results developed in this paper and remain unanswered here.

\section{Normal subloops and quotient loop algebras}

Throughout, $F$ is a field and the loop algebra $F[L]$ is the vector space with basis $L$, with the multiplication of $L$ extended $F$-bilinearly.  No associativity or alternativity assumption is imposed.  An \emph{ideal} of a possibly nonassociative algebra $A$ means a vector subspace $I$ such that $AI+IA\subseteq I$.

We use the kernel definition of normality.  More precisely, a subloop $N$ of $L$ is normal, written $N\normal L$, if it is the kernel of a loop homomorphism; equivalently, the quotient loop $L/N$ is defined.  Standard background on loops, congruences, and inverse-property loops can be found in \cite{Pflugfelder}.

\begin{definition}\label{def:ji}
An infinite loop $L$ is \emph{just-infinite} if every nontrivial normal subloop of $L$ has finite index.  An $F$-algebra $A$ is \emph{just-infinite} if $\dim_F A=\infty$ and every nonzero two-sided ideal of $A$ has finite codimension.
\end{definition}

Let $N\normal L$ and let $q_N:L\to L/N$ be the quotient map.  We define the \emph{relative augmentation ideal} to be the kernel of its linear extension:
\[
  \omega_F(N):=\ker F[q_N].
\]

\begin{proposition}\label{propquotient}
For every loop $L$ and every normal subloop $N\normal L$, we have
\[
  \omega_F(N)
  =\operatorname{span}_F\{x-y:q_N(x)=q_N(y)\}
  =\langle n-1:n\in N\rangle_{F[L]}.
\]
Moreover, there is a natural isomorphism of $F$-algebras
\[
  F[L]/\omega_F(N)\cong F[L/N],
\]
and consequently
\[
  \codim_{F[L]}\omega_F(N)=|L/N|,
\]
where either side is allowed to be infinite.
\end{proposition}

\begin{proof}
The linear extension $F[q_N]$ is a surjective algebra homomorphism, and its kernel is spanned by the differences of basis elements with the same image. Let $J$ be the ideal generated by $\{n-1:n\in N\}$ and see that $J$ is contained in $\ker F[q_N]$.  

Conversely, if $q_N(x)=q_N(y)$, then $n=x\backslash y$ belongs to $N$ and $y=xn$, and then $y-x=x(n-1)\in J.$
Thus $J=\ker F[q_N]=\omega_F(N)$, and the isomorphism follows from the homomorphism theorem.  Since $L/N$ is a basis of $F[L/N]$, the last assertion follows as well.
\end{proof}

\begin{theorem}\label{teoalgebra-implies-loop}
Let $L$ be an infinite loop.  If $F[L]$ is just-infinite, then $L$ is just-infinite.
\end{theorem}

\begin{proof}
Let $1\ne N\normal L$ and choose $1\ne n\in N$.  Then $n-1$ is a nonzero element of $\omega_F(N)$, since distinct elements of $L$ are linearly independent in $F[L]$.  Just-infiniteness of $F[L]$ therefore implies that $\omega_F(N)$ has finite codimension.  Proposition~\ref{propquotient} gives
\[
  |L/N|=\dim_F F[L/N]<\infty.
\]
Thus every nontrivial normal subloop of $L$ has finite index.
\end{proof}

\begin{remark}\label{remconverse}
The converse of Theorem~\ref{teoalgebra-implies-loop} is false.  For example, the first Grigorchuk group is a just-infinite branch group, whereas its complex group algebra is not even $*$-just-infinite \cite[Theorem~7.10]{GMR}.  Since $*$-just-infiniteness tests only $*$-invariant ideals, it is weaker than just-infiniteness in the purely algebraic sense.  Hence the complex group algebra in this example is not just-infinite.
\end{remark}

\section{Faithful ideals and a criterion for just-infinite loop algebras}

We say that a loop $L$ is an \emph{inverse-property loop}, or an IP-loop, if it possesses a map $x\mapsto x^{-1}$ satisfying the identities
\[\begin{array}{rcl}
  x^{-1}(xy)&=&y\\
  (yx)x^{-1}&=&y,
  \end{array}
\]
for all $x,y\in L$.  In this case, the division operations are given by
\[\begin{array}{rcl}
  x\backslash y&=&x^{-1}y\\
  y/x&=&yx^{-1}.
  \end{array}
\]
Every Moufang loop is an inverse-property loop (\cite[IV. 1.4 Theorem]{Pflugfelder}).

Let $L$ be an inverse-property loop and let $I\normal F[L]$.  Define an equivalence relation on $L$ by
\begin{equation}\label{eqcong}
  x\equiv_I y
  \Leftrightarrow
  x-y\in I.
\end{equation}

\begin{lemma}\label{lem:ideal-congruence}
The relation $\equiv_I$ in (\ref{eqcong}) is a loop congruence.  Its identity class
\[
  N_I:=\{x\in L:x-1\in I\}
\]
is a normal subloop of $L$, and $\omega_F(N_I)\subseteq I$.
\end{lemma}

\begin{proof}
Let $x,y,x',y'\in L$.  If $x-x'\in I$ and $y-y'\in I$, then, since $I$ is an ideal, $(x-x')y\in I$ and $x'(y-y')\in I$.  Hence
\[
  xy-x'y'=(x-x')y+x'(y-y')\in I,
\]
which proves compatibility with multiplication.  The relation is also compatible with inversion.  Indeed, the inverse-property identities give
\[
  x^{-1}\bigl((y-x)y^{-1}\bigr)
  =x^{-1}(yy^{-1})-x^{-1}(xy^{-1})
  =x^{-1}-y^{-1}.
\]
Thus $x-y\in I$ implies $x^{-1}-y^{-1}\in I$.  Since the division operations are expressed in terms of multiplication and inversion, $\equiv_I$ is a loop congruence.  Its identity class $N_I$ is therefore normal.

Two loop elements have the same image in $L/N_I$ exactly when they are $\equiv_I$-equivalent.  Hence all the spanning differences of $\omega_F(N_I)$ belong to $I$.
\end{proof}

\begin{theorem}\label{teoadjunction}
Let $L$ be an IP-loop, let $N\normal L$, and let $I\normal F[L]$.  Then $\omega_F(N)\subseteq I$ if and only if $N\subseteq N_I$.  Furthermore, $N_{\omega_F(N)}=N$.  Consequently, $N\mapsto\omega_F(N)$ is an order embedding of the normal subloops of $L$ into the ideals of $F[L]$, and $\omega_F(N_I)$ is the largest relative augmentation ideal contained in $I$.
\end{theorem}

\begin{proof}
If $\omega_F(N)\subseteq I$, then $n-1\in I$ for every $n\in N$, and hence $N\subseteq N_I$.  Conversely, if $N\subseteq N_I$, then $n-1\in I$ for every $n\in N$.  Proposition~\ref{propquotient} shows that these elements generate
$\omega_F(N)$ as an ideal, so $\omega_F(N)\subseteq I$.

For $x\in L$, we have
\[
  x\in N_{\omega_F(N)} \Leftrightarrow
  F[q_N](x-1)=0 \Leftrightarrow q_N(x)=q_N(1).
\]
The last condition is equivalent to $x\in N$, which proves $N_{\omega_F(N)}=N$.  Moreover, for normal subloops $N,M\normal L$, the adjunction gives
\[
  \omega_F(N)\subseteq\omega_F(M)
  \Leftrightarrow
  N\subseteq N_{\omega_F(M)}
  \Leftrightarrow
  N\subseteq M.
\]
Thus $N\mapsto\omega_F(N)$ is an order embedding.  Finally, if $\omega_F(M)\subseteq I$ for a normal subloop $M$,
the adjunction gives $M\subseteq N_I$, whence $\omega_F(M)\subseteq\omega_F(N_I)$.  Thus $\omega_F(N_I)$ is the largest
relative augmentation ideal contained in $I$.
\end{proof}

\begin{definition}
We say that an ideal $I\normal F[L]$ is \emph{$L$-faithful} if $N_I=\{1\}$.  Equivalently, the composite set map
\[
  L\hookrightarrow F[L]\twoheadrightarrow F[L]/I
\]
is injective.
\end{definition}

Indeed, if two loop elements are identified modulo $I$, their left quotient belongs to the identity class $N_I$.  Thus $N_I=\{1\}$ forces the two elements to be equal; the converse is immediate.

\begin{theorem}[Faithful-ideal criterion] \label{teofaithful-criterion}
Let $L$ be an infinite IP-loop.  Then $F[L]$ is just-infinite if and only if both of the following conditions hold:
\begin{itemize}
  \item[(i)] $L$ is just-infinite;
  \item[(ii)] every nonzero $L$-faithful ideal of $F[L]$ has finite codimension.
\end{itemize}
\end{theorem}

\begin{proof}
If $F[L]$ is just-infinite, condition (i) follows from Theorem~\ref{teoalgebra-implies-loop}, while (ii) is immediate from the
definition.

Conversely, assume (i) and (ii), and let $0\ne I\normal F[L]$.  If $N_I=\{1\}$, then $I$ has finite codimension by (ii).  If $N_I\ne\{1\}$, then $L/N_I$ is finite by (i).  Theorem~\ref{teoadjunction} and Proposition~\ref{propquotient} give a surjection
\[
  F[L]/\omega_F(N_I)\twoheadrightarrow F[L]/I
\]
from the finite-dimensional algebra $F[L/N_I]$.  Hence $F[L]/I$ is finite-dimensional, which means that every nonzero ideal has finite codimension.
\end{proof}

\begin{remark}
Notice that Theorem~\ref{teofaithful-criterion} separates the two sources of ideals. The ideal $\omega_F(N_I)$ is the part detected by the normal-subloop lattice; the quotient $I/\omega_F(N_I)$ records additional linear relations.  An $L$-faithful ideal is exactly an ideal for which the first part vanishes.
\end{remark}

\section{Just-infinite Chein loops}

Let $G$ be a group and let $u$ be a symbol not in $G$.  On the set $M(G,2)=G\sqcup Gu$, Chein's multiplication~\cite{Chein} is determined, for $g,h\in G$, by
\begin{equation}\label{eq:chein-loop}
\begin{array}{rcl}
g\cdot h&=&gh,\\
g\cdot(hu)&=&(hg)u,\\
(gu)\cdot h&=&(gh^{-1})u,\\
(gu)\cdot(hu)&=&h^{-1}g,
\end{array}
\end{equation}
It is well known that these rules define a Moufang loop, which is a group if and only if $G$ is abelian. We next show that this construction preserves the just-infinite property of the loop in both directions.

\begin{theorem}\label{teoloop-chein}
For every infinite group $G$, the loop $M(G,2)$ is just-infinite if and only if $G$ is just-infinite.
\end{theorem}

\begin{proof}
Suppose first that $M=M(G,2)$ is just-infinite.  If $N\normal G$, then the natural map
\[
  M(G,2)\rightarrow M(G/N,2)
\]
has kernel $N$.  Thus every nontrivial normal subgroup of $G$ is a nontrivial normal subloop of $M$ and consequently has finite index in $M$, hence also in $G$.  Therefore $G$ is just-infinite.

Conversely, assume that $G$ is just-infinite, and let $1\ne Q\normal M$. Putting $H=Q\cap G$, the restriction to $G$ of the quotient map $\pi:M\to M/Q$ has kernel $H$, so $H\normal G$.  If $H\ne1$, then $G/H$ is
finite.  Since
\[
  \pi(M)=\pi(G)\cup\pi(G)\pi(u),
\]
the loop $M/Q$ is finite.

It remains to rule out $H=1$.  In that case a nonidentity element of $Q$ must have the form $gu$.  Hence $\pi(u)=\pi(g)^{-1}\in\pi(G)$, so $\pi(M)=\pi(G)$ and $\pi|_G$ is injective.  The identity
\[
  a(bu)=(ba)u,\text{ for } a,b\in G
\]
then shows, after applying $\pi$, that $ab=ba$ for all $a,b\in G$.  Indeed, $\pi(u)\in\pi(G)$, so the calculation takes place in the associative subgroup $\pi(G)$, where cancellation of $\pi(u)$ is legitimate.  Thus $G$ is abelian.  The identity $uh=h^{-1}u$ and the fact that $\pi(u)\in\pi(G)$ now give $h=h^{-1}$ for every $h\in G$.  Hence $G$ is an infinite elementary abelian $2$-group.  Such a group is not just-infinite, because every subgroup of order $2$ is normal and has
infinite index.  This contradiction proves that $H\ne1$, and the preceding paragraph completes the proof.
\end{proof}

\begin{corollary}\label{corloop-examples}
If $G$ is an infinite nonabelian just-infinite group, then $M(G,2)$ is a nonassociative just-infinite Moufang loop.
\end{corollary}

\section{Ideals in the standard Chein loop algebra}

Fix a field $F$ and write $A=F[G]$ and $B=F[M(G,2)]=A\oplus Au$.  Let $*:A\to A$ be the linear extension of $g\mapsto g^{-1}$; this is an involutory anti-automorphism.  Extending the multiplication rules of $M(G,2)$ bilinearly gives, for $a,b\in A$,
\begin{equation}\label{eq:chein-algebra}
\begin{array}{rcl}
  a(bu)&=&(ba)u,\\
  (au)b&=&(ab^*)u,\\
  (au)(bu)&=&b^*a.\\
\end{array}
\end{equation}

If $K\normal A$, then
$K^\ast=\{k^\ast:k\in K\}\normal A$.

Recall that an associative algebra is prime if the product of any two nonzero ideals is nonzero.  We shall use Connell's primeness criterion \cite{Connell}, which states that, for a field $F$, the group algebra $F[G]$ is prime if and only if $G$ has no nontrivial finite normal subgroup.  In particular, if $G$ is infinite and just-infinite, then $F[G]$ is prime.

\begin{proposition}\label{propdoubled-ideals}
Let $K\normal A$.  Then $\widetilde K:=K\oplus Ku$ is an ideal of $B$ if and only if $K^*=K$.  In this case
\[
  B/\widetilde K\cong (A/K)\oplus(A/K)u
\]
in the evident algebraic sense, namely as the Chein-type algebra built from $A/K$ and its induced involutory anti-automorphism, and
\[
  \dim_F(B/\widetilde K)=2\dim_F(A/K).
\]
Moreover, if $J\normal B$ and $K=J\cap A$, then $K^*=K$ and $\widetilde K\subseteq J$.
\end{proposition}

\begin{proof}
If $K^*=K$, the multiplication rules in \eqref{eq:chein-algebra} show directly that $K\oplus Ku$ is stable under left and right multiplication by $A\oplus Au$.  Hence $\widetilde K\normal B$.

Conversely, suppose that $\widetilde K=K\oplus Ku$ is an ideal of $B$.  For $k\in K$, we have $ku\in Ku\subseteq\widetilde K$ and hence $u(ku)\in\widetilde K$.  Since
\[
  u(ku)=(1u)(ku)=k^*1=k^\ast,
\]
we have $k^\ast\in\widetilde K$.  Since $k^\ast\in A$,
\[
  k^\ast\in\widetilde K\cap A=(K\oplus Ku)\cap A=K.
\]
Thus $K^\ast\subseteq K$, and equality follows by applying $*$ again.  The quotient and dimension statements follow from $B=A\oplus Au$.

Now let $J\normal B$ and $K=J\cap A$.  For $k\in K$, first $ku\in J$ and then $u(ku)=k^*\in J\cap A=K$.  Thus $K^*=K$.  The same multiplication by $u$ shows $Ku\subseteq J$, and therefore $K\oplus Ku\subseteq J$.
\end{proof}

\begin{theorem}\label{teoideal_intersection}
Let $G$ be infinite and suppose that $A=F[G]$ is prime.  Then every nonzero ideal $J\normal B=F[M(G,2)]$ satisfies
\[
  J\cap A\ne0.
\]
Consequently, every nonzero ideal of $B$ contains a nonzero doubled ideal $K\oplus Ku$ with $K^*=K$.
\end{theorem}

\begin{proof}
Assume, for a contradiction, that $0\ne J\normal B$ and $J\cap A=0$.  We also have $J\cap Au=0$, since $(au)u=(au)(1u)=a$ for $a\in A$.  Let $p_A:B\to A$ be the projection and put $V=p_A(J)$.  The restriction $p_A|_J$ is injective, so there is a unique linear map $T:V\to A$ such that
\[
  J=\{a+T(a)u:a\in V\}.
\]
Multiplication on either side by elements of $A$, followed by comparison of the two components, shows that $V\normal A$ and, for all $a\in V$ and $c\in G$, we have
\begin{equation}\label{eq:T-standard}
 \begin{array}{rcl}
 T(ca)&=&T(a)c,\\
  T(ac)&=&T(a)c^{-1}.\\
\end{array}
\end{equation}
Multiplication on the right by $u$ gives
\[
  (a+T(a)u)u=T(a)+au.
\]
Comparison with the graph description shows that $T(V)=V$ and $T^2=1$ on $V$.  In particular, $V\ne0$.

For $a\in V$ and $c,d\in G$, calculate $T(cad)$ in two ways using
\eqref{eq:T-standard}:
\[\begin{array}{rcl}
  T((ca)d)&=&T(a)cd^{-1},\\
  T(c(ad))&=&T(a)d^{-1}c.
  \end{array}
\]
Since $T(V)=V$, it follows that, for $c,d\in G$,
\[
  V(cd^{-1}-d^{-1}c)=0.
\]
If $G$ is nonabelian, let $C$ be the ideal of $A$ generated by the elements $cd^{-1}-d^{-1}c$.  Then $C\ne0$, whereas the preceding identity and the fact that $V$ is an ideal imply $VC=0$.  This contradicts the primeness of $A$.

Suppose instead that $G$ is abelian.  Since $ca=ac$, the two identities in \eqref{eq:T-standard} give
\[
  V(c-c^{-1})=0.
\]
If $c-c^{-1}$ were nonzero for some $c$, the nonzero ideal it generates would be annihilated on the left by $V$, contrary to primeness.  Hence $c=c^{-1}$ for every $c\in G$.  But then $G$ is an elementary abelian
$2$-group, and any subgroup of order $2$ is a nontrivial finite normal subgroup.  Connell's criterion contradicts the primeness of $A$ because $G$ is infinite.  Thus $J\cap A\ne0$.

For the final assertion, take $K=J\cap A$.  Proposition \ref{propdoubled-ideals} gives the nonzero doubled ideal
$K\oplus Ku\subseteq J$.
\end{proof}

\begin{theorem}\label{teochein-main}
For every infinite group $G$ and every field $F$, the algebra $F[M(G,2)]$ is just-infinite if and only if $F[G]$ is just-infinite.
\end{theorem}

\begin{proof}
Assume first that $A=F[G]$ is just-infinite.  By Theorem~\ref{teoalgebra-implies-loop}, $G$ is just-infinite, and hence $A$
is prime by Connell's criterion.  Let $0\ne J\normal B$.  By Theorem~\ref{teoideal_intersection}, $K=J\cap A$ is nonzero.  Since $A$ is just-infinite, $A/K$ is finite-dimensional.  Since $K\oplus Ku\subseteq J$, Proposition~\ref{propdoubled-ideals} gives
\[
  \dim_F\bigl(B/(K\oplus Ku)\bigr)
  =2\dim_F(A/K)<\infty.
\]
Therefore $B/J$ is finite-dimensional, and $B$ is just-infinite.

Conversely, assume that $B$ is just-infinite.  Theorem \ref{teoalgebra-implies-loop} and Theorem~\ref{teoloop-chein} imply that
$G$ is just-infinite, hence $A$ is prime.  Let $0\ne K\normal A$.  Since $A$ is prime, the product $H=KK^*$ is nonzero. Furthermore, $H$ is an ideal, satisfies $H^*=H$, and is contained in $K$. Proposition~\ref{propdoubled-ideals} shows that $H\oplus Hu$ is a nonzero ideal of $B$.  It therefore has finite codimension, and the dimension formula in that proposition gives
\[
  2\dim_F(A/H)=\dim_F\bigl(B/(H\oplus Hu)\bigr)<\infty.
\]
Since $H\subseteq K$, the quotient $A/K$ is finite-dimensional, therefore $A$ is just-infinite.
\end{proof}

The following notion was first introduced and systematically investigated by Bell, Farina, and Pendergrass--Rice \cite[Definition~1.2]{BellFarinaRice}.

\begin{definition}
A just-infinite $F$-algebra $A$ is called \emph{stably just-infinite over $F$} if $E\otimes_F A$ is just-infinite as an $E$-algebra for every field extension $E/F$.
\end{definition}

\begin{corollary}[Scalar stability]\label{corstable-standard} For every infinite group $G$, the $F$-algebra $F[M(G,2)]$ is stably just-infinite over $F$ if and only if $F[G]$ is stably just-infinite over $F$.
\end{corollary}

\begin{proof}
For every extension $E/F$ there are natural isomorphisms
\[\begin{array}{rcl}
  E\otimes_F F[G]&\cong& E[G]\\
  E\otimes_F F[M(G,2)]&\cong& E[M(G,2)].
\end{array}
\]
Applying Theorem~\ref{teochein-main} over $E$ completes the proof.
\end{proof}

\section{Generalized Chein doubles}

We now consider Chein's generalized construction \cite{Chein1978}, in the formulation recalled and developed by Kinyon, Phillips, and Vojt\v{e}chovsk\'y \cite{KPV}.  Let $G$ be a group, let $*:G\to G$ be an involutory anti-automorphism, and choose
$g_0\in Z(G)$ such that, for all $g\in G$,
\begin{equation}\label{admissible}
  \begin{array}{l}
  g_0^*=g_0,\\
  gg^*\in Z(G).\\
  \end{array}
\end{equation}
The multiplication on $M(G,*,g_0)=G\sqcup Gu$ is given by the rules
\begin{equation}\label{eq:generalized-loop}
\begin{array}{rcl}
  g\cdot h&=&gh,\\
  g\cdot(hu)&=&(hg)u,\\
  (gu)\cdot h&=&(gh^*)u,\\
  (gu)\cdot(hu)&=&g_0h^*g.
\end{array}
\end{equation}
Under \eqref{admissible}, these rules define a Moufang loop \cite{Chein1978,KPV}, and we call the triple $(G,*,g_0)$ \emph{admissible}. The standard Chein loop is recovered by taking $g^*=g^{-1}$ and $g_0=1$.

Extend $*$ linearly to an involutory anti-automorphism of $A=F[G]$ and put
\[
  B=F[M(G,*,g_0)]=A\oplus Au.
\]
The induced multiplication is, for $a,b\in A$,
\begin{equation}\label{generalized-algebra}
\begin{array}{rcl}
  a(bu)&=&(ba)u,\\
  (au)b&=&(ab^*)u,\\
  (au)(bu)&=&g_0b^*a.
\end{array}
\end{equation}

The preceding arguments extend to generalized Chein doubles.

\begin{proposition}[Generalized doubled ideals]
\label{propgeneralized-doubled}
Let $K\normal A$.  Then $K\oplus Ku$ is an ideal of $B$ if and only if
$K^*=K$.  In this case
\[
  \dim_F\bigl(B/(K\oplus Ku)\bigr)=2\dim_F(A/K).
\]
Moreover, if $J\normal B$ and $K=J\cap A$, then $K^*=K$ and
$K\oplus Ku\subseteq J$.
\end{proposition}

\begin{proof}
If $K^*=K$, the rules in \eqref{generalized-algebra} show that $K\oplus Ku$ is an ideal.  Conversely, if it is an ideal, then
\[
  u(ku)=g_0k^*\in K,
\]
for all $k\in K$.  Multiplication by the unit $g_0^{-1}$ gives $K^*\subseteq K$, and equality follows on applying $*$ again.  The codimension formula follows from the direct-sum decomposition.

Finally, if $J\normal B$ and $K=J\cap A$, then $ku\in J$ for $k\in K$, so $g_0k^*=u(ku)\in J\cap A$.  Multiplying by $g_0^{-1}$ shows that $K^*\subseteq K$, and equality follows on applying $*$ again.  Moreover, $Ku\subseteq J$, and therefore $K\oplus Ku\subseteq J$.
\end{proof}

\begin{theorem}[Generalized ideal intersection] \label{teogeneralized-intersection}
Let $(G,*,g_0)$ be admissible.  If $G$ is infinite and nonabelian and $A=F[G]$ is prime, then every nonzero ideal of
$B=F[M(G,*,g_0)]$ meets $A$ nontrivially.
\end{theorem}

\begin{proof}
Suppose that $0\ne J\normal B$ and $J\cap A=0$.  Then $J\cap Au=0$, because $(au)u=g_0a$.  As in the proof of Theorem
\ref{teoideal_intersection}, $J$ is the graph
\[
  J=\{a+T(a)u:a\in V\}
\]
of a linear map $T:V\to A$, where $0\ne V\normal A$.  Left and right multiplication by $c\in G$ give
\begin{equation}\label{T-generalized}
\begin{array}{rcl}
  T(ca)&=&T(a)c,\\
  T(ac)&=&T(a)c^*,
  \end{array}
\end{equation}
for any $a\in V$.  Right multiplication by $u$ gives $ g_0T(a)\in V$ and $T(g_0T(a))=a.$
Since $g_0$ is a unit and $V$ is an ideal, the first relation implies $T(V)\subseteq V$, while the second implies that $T:V\to V$ is surjective. Thus $T(V)=V$.

For $a\in V$ and $c,d\in G$, associativity in $A$ and \eqref{T-generalized} yield
\[\begin{array}{rcl}
  T((ca)d)&=&T(a)cd^*\\
  T(c(ad))&=&T(a)d^*c.
\end{array}
\]
Since $T(V)=V$ and $*$ is surjective, $V$ annihilates every commutator of $A$ on the right.  Let $C$ be the ideal of $A$ generated by these commutators.  Since $G$ is nonabelian, $C\ne0$, while the preceding identity and the fact that $V$ is an ideal give $VC=0$.  This contradicts the primeness of $A$, and therefore $J\cap A\ne0$.
\end{proof}

\begin{theorem}[Generalized doubling theorem]
\label{teogeneralized-main}
Let $(G,*,g_0)$ be admissible, with $G$ infinite and nonabelian.  Then, for every field $F$, the algebra $F[M(G,*,g_0)]$ is just-infinite if and only if $F[G]$ is just-infinite.
\end{theorem}

\begin{proof}
As usual, let $A=F[G]$ and $B=F[M(G,*,g_0)]$, and suppose first that $A$ is just-infinite.  Then $G$ is just-infinite by
Theorem~\ref{teoalgebra-implies-loop}, so $A$ is prime by Connell's criterion.  If $0\ne J\normal B$, Theorem
\ref{teogeneralized-intersection} gives $0\ne K=J\cap A$. Proposition~\ref{propgeneralized-doubled} shows that $K^*=K$ and
$K\oplus Ku\subseteq J$.  Since $A/K$ is finite-dimensional, so is $B/J$.

Conversely, suppose that $B$ is just-infinite.  We first prove that $A$ is prime.  If $G$ had a nontrivial finite normal subgroup $N$, then $P=NN^*$ would be a nontrivial finite normal subgroup of $G$ with $P^*=P$.  Indeed, $N^*$ is normal, since for $x\in G$ and $n\in N$,
\[
  xn^*x^{-1}
  =\bigl((x^{-1})^*n((x^{-1})^*)^{-1}\bigr)^*\in N^*.
\]
Moreover, two normal subgroups multiply to a normal subgroup, and $(NN^*)^*=NN^*$.
The involution therefore descends to $G/P$, and the natural map
\[
  M(G,*,g_0)\rightarrow
  M(G/P,\overline{*},\overline{g_0})
\]
has kernel $P$.  Since $G$ is infinite and $P$ is finite, this kernel is a nontrivial normal subloop of infinite index, which contradicts Theorem~\ref{teoalgebra-implies-loop}.  Thus $G$ has no nontrivial finite normal subgroup, and $A$ is prime by Connell's criterion.

Now take $0\ne K\normal A$.  Primeness gives $H=KK^*\ne0$.  Moreover, $H^*=H$ and $H\subseteq K$.  Proposition~\ref{propgeneralized-doubled} shows that $H\oplus Hu$ is a nonzero ideal of $B$.  It has finite
codimension, so $A/H$, and therefore $A/K$, is finite-dimensional.  Hence $A$ is just-infinite.
\end{proof}

\begin{corollary}\label{cor:stable-generalized}
Under the hypotheses of Theorem~\ref{teogeneralized-main}, $F[M(G,*,g_0)]$ is stably just-infinite over $F$ if and only if $F[G]$ is stably just-infinite over $F$.
\end{corollary}

\begin{proof}
Apply Theorem~\ref{teogeneralized-main} after every field extension $E/F$, using the natural scalar-extension isomorphisms.
\end{proof}

\section{Uniform examples and finiteness properties}

Recall that a possibly nonassociative $F$-algebra $C$ is \emph{residually finite-dimensional} if the intersection of its ideals of finite codimension is zero.  It is \emph{locally finite-dimensional} if every finitely generated subalgebra is finite-dimensional.

\begin{proposition}\label{prop:finite-properties}
Let $G$ be a group and $L=M(G,2)$.
\begin{itemize}
  \item[(i)] If $G$ is locally finite, then $L$ is locally finite and $F[L]$ is  locally finite-dimensional for every field $F$.
  \item[(ii)] If $G$ is residually finite, then $L$ is residually finite.  For   every field $F$, the algebra $F[L]$ is residually finite-dimensional.
\end{itemize}
\end{proposition}

\begin{proof}
For (i), a finite subset of $L$ involves only finitely many elements of $G$. They generate a finite subgroup $H$ of $G$, and the given subset is contained in the finite subloop $M(H,2)$.  Similarly, finitely many elements
of $F[L]$ have finite total support, so the algebra they generate is contained in the finite-dimensional algebra $F[M(H,2)]$.

For (ii), two distinct elements of $L$ either lie in different components $G$ and $Gu$, or their $G$-coordinates can be separated in a finite quotient of $G$.  The corresponding map onto a finite Chein loop separates them.
Thus $L$ is residually finite.  Now let $0\ne\alpha\in F[L]$.  Its support is finite.  Taking a diagonal product of finitely many separating maps gives a homomorphism from $L$ to a finite loop that is injective on this support.  The
induced homomorphism of loop algebras has finite-dimensional image and does not annihilate $\alpha$.  Therefore $F[L]$ is residually finite-dimensional.
\end{proof}

Belyaev, Grigorchuk, and Shumyatsky constructed a single infinite, locally finite, residually finite group $G$ such that $F[G]$ is just-infinite for every field $F$ \cite[Theorem~2.4]{BGS}.  Combining their construction with
the preceding results gives the following uniform nonassociative example.

\begin{theorem}\label{teouniform-example}
There exists an infinite, locally finite, residually finite, nonassociative Moufang loop $L$ with the following properties.
\begin{itemize}
  \item[(i)] The loop $L$ is just-infinite.
  \item[(ii)] For every field $F$, the loop algebra $F[L]$ is nonassociative,
  just-infinite, locally finite-dimensional, and residually
  finite-dimensional.
  \item[(iii)] For every field $F$, the algebra $F[L]$ is stably just-infinite over
  $F$.
\end{itemize}
\end{theorem}

\begin{proof}
Take the group $G$ of \cite[Theorem~2.4]{BGS} and put $L=M(G,2)$.  The group $G$ is necessarily nonabelian.  Indeed, an infinite locally finite abelian group has a nontrivial finite subgroup of infinite index and so cannot be
just-infinite, whereas Theorem~\ref{teoalgebra-implies-loop} shows that this $G$ is just-infinite.  Hence $L$ is a nonassociative Moufang loop.

Theorem~\ref{teochein-main} implies that $F[L]$ is just-infinite for every field $F$, and Theorem~\ref{teoloop-chein} gives (i).
Proposition~\ref{prop:finite-properties} supplies the stated local and residual properties.  Finally, if $E/F$ is any field extension, then
\[
  E\otimes_F F[L]\cong E[L],
\]
and $E[L]$ is just-infinite by the same every-field property, which proves
(iii).
\end{proof}

\section{The obstruction from RA loops}

Recall that an RA loop is a loop whose loop ring over a suitable commutative coefficient ring is alternative but nonassociative.  A basic structure theorem for nonassociative RA loops states that their commutator--associator
subloop is
\[
  L'=\{1,s\}\subseteq Z(L)
\]
for an element $s$ of order $2$ \cite[Theorem~IV.3.1]{GoodaireJespersMilies}; see also \cite[Theorem~2.1]{CornelissenMilies}.

\begin{corollary}\label{corra-obstruction}
No infinite nonassociative RA loop is just-infinite.  Consequently, if $L$ is such a loop, then $F[L]$ is not just-infinite over any field $F$. In particular, if $\operatorname{char}F\ne2$ and $F[L]$ is alternative but
nonassociative for an infinite loop $L$, then $F[L]$ is not just-infinite.
\end{corollary}

\begin{proof}
The nontrivial normal subloop $L'=\{1,s\}$ is finite and hence has infinite index in $L$, therefore $L$ is not just-infinite.  The remaining assertions follow from Theorem~\ref{teoalgebra-implies-loop} and the definition of an RA loop.
\end{proof}

\section{Further questions}

The preceding results reduce the standard Chein case to the corresponding group-algebra problem. Nevertheless several natural directions remain open, and we list three of them below.

\begin{question}
For which classes of inverse-property loops does just-infiniteness of $L$ force every nonzero $L$-faithful ideal of $F[L]$ to have finite codimension?
\end{question}

\begin{question}
Does a loop-level analogue of Theorem~\ref{teogeneralized-main} hold for generalized Chein doubles?  More precisely, which admissible triples satisfy
\[
  M(G,*,g_0)\text{ just-infinite}
  \Leftrightarrow
  G\text{ just-infinite}?
\]
\end{question}

\begin{question}
Is there a finitely generated nonassociative Moufang loop $L$ and a
field $F$ such that $F[L]$ is just-infinite? Can such an example be
chosen so that $F[L]$ is stably just-infinite over $F$? More strongly,
can one choose a single loop $L$ for which $F[L]$ is just-infinite for
every field $F$?
\end{question}

\end{document}